\documentclass{aims}
\usepackage[numbers]{natbib}
\usepackage{float}
\usepackage{textcomp}
\usepackage{amsthm}
\usepackage{amssymb}
\usepackage{graphicx}
\usepackage{algorithm}
\usepackage{algorithmic}
\usepackage{multirow}
\usepackage{booktabs}
\usepackage{rotating}
\usepackage{amsmath}
  \usepackage{paralist}
  \usepackage{graphics} %% add this and next lines if pictures should be in esp format
  \usepackage{epsfig} %For pictures: screened artwork should be set up with an 85 or 100 line screen
\usepackage{graphicx}  \usepackage{epstopdf}%This is to transfer .eps figure to .pdf figure; please compile your paper using PDFLeTex or PDFTeXify.
 \usepackage[colorlinks=true]{hyperref}
\hypersetup{urlcolor=blue, citecolor=red}

\def\currentvolume{X}
 \def\currentissue{X}
  \def\currentyear{200X}
   \def\currentmonth{XX}
    \def\ppages{X--XX}
     \def\DOI{10.3934/xx.xx.xx.xx}

\newtheorem{theorem}{Theorem}[section]

\theoremstyle{definition}

\newtheorem{property}{\bf Property}[section]

\title[GVNS for single-machine scheduling problems with step-deterioration]
      {A general variable neighborhood search for single-machine total
tardiness scheduling problems with step-deteriorating jobs}

\author[Peng Guo and Wenming Cheng and Yi Wang]{}


 \email{pengguo318@gmail.com}
 \email{wmcheng@home.swjtu.edu.cn}
 \email{ywang2@aum.edu}

\thanks{}

\begin{document}
\maketitle

% Enter the first author's name and address:
\centerline{\scshape Peng Guo and  Wenming Cheng}
\medskip
{\footnotesize
% please put the address of the first author
 \centerline{School of Mechanical Engineering}
   \centerline{Southwest Jiaotong University, Chengdu, China}
} % Do not forget to end the {\footnotesize by the sign }

\medskip

\centerline{\scshape Yi Wang}
\medskip
{\footnotesize
 % please put the address of the second  and third author
 \centerline{ Department of Mathematics}
   \centerline{Auburn University at Montgomery, AL, USA}
}

\bigskip

% The name of the associate editor will be entered by an editorial staff
% "Communicated by the associate editor name" is not needed for special issue.
 \centerline{(Communicated by the associate editor name)}

%The abstract of your paper
\begin{abstract}
In this article, we study a single-machine scheduling problem with the objective of minimizing the total tardiness for a set of independent jobs. The processing time of a job is modeled as a step function of its starting time and a specific deteriorating date. The total tardiness as one important objective in practice has not been concerned in the studies of single-machine scheduling problems with step-deteriorating jobs.     To overcome the NP-hardness of the problem, we propose a heuristic named simple weighted search procedure (SWSP)  and a general variable neighborhood search algorithm (GVNS) to obtain near optimal solutions.  Extensive numerical experiments are carried out on   randomly generated test instances in order to evaluate the performance of the proposed algorithms. By comparing to   the CPLEX optimization solver, the heuristic SWSP and the standard variable neighborhood search, it is shown that the  proposed GVNS algorithm   can provide better solutions within a reasonable running time.
\end{abstract}

%The title of your section 1
\section{Introduction}

As an important tool for manufacturing and engineering, scheduling
has a major impact on the productivity of a process. In the classical
model of scheduling theory, it is assumed that the processing times
of all jobs are known in advance and remain constant during the entire
production process. However, this assumption may not  be  applicable
to model some real manufacturing and service problems. Examples can be
found in steel production, equipment maintenance, cleaning tasks allocation
and so forth \citep{Kunnathur1990SDPLD,Mosheiov1991}, where any delay
or waiting in starting to process a job may increase the necessary time
for its completion. Such kinds of jobs are called {\em deteriorating jobs}.
This kind of {\em scheduling problems with deteriorating jobs} was firstly considered by
\citet{Browne1990} and \citet{Gupta1988}. They assumed that the
processing time of a job is a  single linearly non-decreasing function of its
starting time and aimed at minimizing the expected makespan and the
variance of the makespan in a single-machine environment. The function that describes the processing time of a deteriorating job shall henceforth be  called a {\em deterioration function}. Since then,
there is a rapidly growing interest in the literature to study various
types of scheduling problems involving deteriorating jobs. A recent
survey of scheduling problems with deteriorating jobs was provided
by \citet{Cheng2004survey}.

Most of the research concentrated on obtaining the optimal schedule
for scheduling   jobs for which their processing times continuously depend on their starting times  \citep{Bachman2000LD,Oron2007SL,Wu2007LD,Ji2009SL,Wang2010SL,Jafari2012LD}.
However, if some jobs fail to be processed prior to a pre-specified
threshold, their processing times will be extended by adding an extra
penalty time in some situations. Job processing times of this type of  deterioration
may be characterized  by a piecewise defined     function.
\citet{Kunnathur1990SDPLD} firstly studied single-machine scheduling
with a piecewise deterioration function. They assumed the actual processing
time of a job is split into a fixed part and a   variable penalty  part.
The variable part depends linearly upon the starting time of the job.
\citet{Sundararaghavan1994SDPLD} introduced a total weighted completion
time scheduling problem with step deterioration. Then, \citet{Mosheiov1995SDPLD}
investigated  the scheduling problem with step deterioration in which the objective is to minimize the
makespan on a single machine or  multiple machines. He also  introduced some simple
heuristics for all these NP-hard problems.

For the scheduling problem with a piecewise linear deterioration function,
\citet{Kubiak1998SDPLD} considered single-machine scheduling with
a generalization of the unbounded deterioration model proposed by
\citet{Browne1990} and presented NP-hardness proofs. They also developed a
pseudo-polynomial dynamic programming algorithm and a branch and bound
algorithm to obtain the optimal schedule. Alternatively, \citet{Cheng2003SDPLD}
studied the problem of scheduling jobs with piecewise linear decreasing
processing times on a single or multiple machines. They aimed to minimize the
makespan and the {\em flow time}, i.e., the total completion time,  and proved that the two problems
are NP-hard. Subsequently, \citet{Ji2007SDPLD} gave a fully polynomial
time approximation scheme for the same case with a single-machine.
Afterwards, \citet{Wu2009SDPLD} provided two heuristic algorithms
to solve the single-machine problem under the piecewise linear deterioration
model. \citet{Moslehi2010SDPLD} dealt with the single-machine scheduling
problem with the assumption of piecewise-linear deterioration. They suggested a branch
and bound algorithm and a heuristic algorithm with complexity $O(n^{2})$ for
the objective of minimizing the total number of tardy jobs, where $n$ is the total number of jobs in the problem.

Some articles focused on the scheduling problem with step-deterioration
have been published. \citet{Cheng2001SDPLD} studied some single-machine
scheduling problems and showed that the {\em flow time problem} is NP-complete.
\citet{Jeng2004SDPLD} introduced a branch and bound algorithm for
the single-machine problem of minimizing the makespan. In addition,
They also studied the same problem for minimizing the flow time \citep{Jeng2005SDPLD}. \citet{He2009SDPLD} proposed an exact
algorithm to solve most of the problems with up to 24 jobs and a heuristic
algorithm to derive a near-optimal solution. \citet{Layegh2009SDPLD}
minimized the total weighted completion time by using a memetic algorithm.
\citet{Cheng2012SDPLD} developed a variable neighborhood search to
minimize the flow time on parallel machines. Furthermore,
batch scheduling with step-deterioration has got attentions \citep{Leung2008SDPLD,Mor2012SDPLD}.

However, {\em the total tardiness as one important objective in practice
has not been concerned in the studies of single-machine scheduling
problem with step-deterioration}. Minimizing {\em total tardiness} on a single machine has been drawing considerable attentions
of  researchers in the past decades and  it is  NP-hard in the ordinary sense when no deterioration is considered \citep{DU1990NP}. The latest theoretical developments for the
single-machine scheduling problem and the current state-of-the-art
algorithms are reviewed by \citet{Koulamas2010RW}.
He found that the single-machine total tardiness scheduling problem
continues to attract researchers' interests from both theoretical and
practical perspectives. {\em In this article, we consider the single-machine
total tardiness scheduling problem with step-deteriorating jobs}. Since
the total tardiness problem without step-deterioration in a single-machine is a NP-hard problem, the problem tackled here is also a NP-hard
problem. To the best of our knowledge, there is no literature on minimizing
the total tardiness in a single-machine scheduling problem with step-deteriorating
jobs.

The remainder of the study is organized as follows. Section \ref{sec:Sec2}
contains the problem description and a mixed integer programming
model. In Section \ref{sec:Sec3}, a heuristic and a general
variable neighborhood search algorithm are developed for the proposed
problem. The    proposed methods are tested  and compared to the CPLEX and the variable neighborhood search (VNS) on variance test instances of both small size and large size
in Section \ref{sec:Sec4}. In the last section, conclusions and future
works are mentioned.

\section{Problem formulation\label{sec:Sec2}}

The single-machine total tardiness scheduling problem with step deteriorating
jobs
can be described as follows.   Let  $N:=\left\{ 1,2,\text{\ldots},n\right\} $ be the set of $n$ jobs
  to be scheduled on a single-machine without
preemption. Assume that all jobs are ready at time zero and the machine
is available at all times. In addition, the machine can handle only
one job at a time, and cannot keep idle until the last job assigned
to it is processed and finished. For each job  $j\in N$, there is a {\em basic
processing time} $a_{j}$, a {\em due date} $d_{j}$ and a given {\em deteriorating threshold}, also called {\em deteriorating
date} $h_{j}$. If the  {\em starting time} $s_{j}$ of job $j\in N$ is less than or
equal to the given threshold $h_{j}$, then job $j$ only requires a basic
processing time $a_{j}$. Otherwise,   an extra
penalty $b_{j}$ is incurred. Thus, the {\em actual processing time} $p_{j}$ of job
$j$ can be defined as a step-function: $p_{j}=a_{j}$ if $s_{j}\leqslant h_{j}$;
$p_{j}=a_{j}+b_{j}$, otherwise. Without loss of generality, the four parameters $a_{j}$, $b_{j}$,
$d_{j}$ and $h_{j}$ are assumed to be integers.

Let $\pi=(\pi_{1},\ldots,\pi_{n})$
be a sequence that arranges the current processing order of jobs in $N$, where $\pi_{k}$, $k=1,\ldots,n$,
indicates  the job in position $k$.
The {\em tardiness} $T_{j} $
of a job $j$ in a schedule $\pi$ can be calculated by $T_{j} =\max\left\{ 0,s_{j} +p_{j} \text{\textminus}d_{j}\right\} $.
The objective is to find a schedule $\pi^{*}$ such that the {\em total tardiness}
$\sum T_{j}$ is minimized. The total tardiness is a non-decreasing criterion
of the job completion times. Using the three-field notation, this
problem can be denoted by $1|p_{j}=a_{j}$ or $a_{j}+b_{j},h_{j}|\sum T_{j}$.

Based on the  above description, we   formulate the problem as a  0-1
integer programming model. Firstly, the decision variable $y_{ij}$, $i,j \in N$ is defined
 such that  $y_{ij}$ is 1 if   job $i$ is scheduled before job
$j$ on the single-machine, and 0 otherwise. For each pair of jobs $i$ and $j$, $y_{ij}+y_{ji}=1$. Then, for a schedule $\pi$ of the jobs in $N$, we minimize
\begin{gather}
Z(\pi):=\sum_{j=1}^{n}T_{j}\label{eq:2.1}
\end{gather}
subject to
\begin{eqnarray}
p_{j}&=&\begin{cases}
a_{j}, & s_{j}\leqslant h_{j}\\
a_{j}+b_{j}, & \mathrm{otherwise},
\end{cases}\qquad \begin{gathered}\forall j\in N\end{gathered}
\label{eq:2.2}\\
s_{i}+p_{i}&\leqslant & s_{j}+M(1-y_{ij}),\qquad\forall i,j\in N,i\ne j\label{eq:2.3}\\
%s_{j}+p_{j} &\leqslant & s_{i}+My_{ij},\qquad\forall i,j\in N,i<j\label{eq:2.4}\\
s_{j}+p_{j}-d_{j} &\leqslant & T_{j},\qquad\forall j\in N\label{eq:2.5}\\
y_{ij} &\in & \left\{ 0,1\right\}, \qquad\forall i,j\in N,i\neq j\label{eq:2.6}\\
s_{j},T_{j}& \geqslant & 0,\qquad\forall j\in N, \label{eq:2.7}
\end{eqnarray}
where $M$ is a large positive constant such that $M\rightarrow\infty$ as $n\to \infty$.
For example,  $M$ may be  chosen as  $M:=\max_{j\in N}\left\{ d_{j}\right\} +\sum_{j\in N}(a_{j}+b_{j})$
.

In the above mathematical model, equation  (\ref{eq:2.1}) represents
the objective of minimizing  the total tardiness. Constraint (\ref{eq:2.2})
defines the processing time of each  job.
%If the
%starting time $s_{j}$ of job $j$ is later than its given threshold
%$h_{j}$, its processing time must be penalized and extends to $a_{j}+b_{j}$.
%On the other hand, when the job starts processing before its threshold,
%its processing time $p_{j}$ equals to its basic processing time.
Constraint (\ref{eq:2.3})   determines the starting
time $s_{j}$ of job $j$ with respect to the decision variables $y_{ij}$.
   Constraint
(\ref{eq:2.5}) defines the tardiness of job $j$.
Finally, constraints (\ref{eq:2.6}) and (\ref{eq:2.7}) define the
boundary values of variables $y_{ij}$, $s_j$, $T_j$, for $i,j\in N$.

Next theorem concerns the complexity of the problem. 
\begin{theorem}\label{theorem_NP_complete}
The problem
 1 $|p_{j}=a_{j}$ or $ a_{j}+b_{j},h_{j}| \sum T_{j}$
 is \emph{NP-complete}.
\end{theorem}
\begin{proof}
Cheng and Ding [4]  proved that the problem $1|p_i=a_i$ or $a_i+b_i, h_i|\sum C_i$,    that is, the problem of  minimizing the total completion time on a single machine with step-deteriorating jobs, is NP-complete. In our problem, if the due date $d_i$, $i\in N$, is set to 0, then the total tardiness scheduling problem is reduced to the total completion time scheduling problem $1|p_i=a_i$ or $a_i+b_i, h_i|\sum C_i$.  Thus, our problem, as a more general case, must be  NP-complete as-well.
\end{proof}

\section{Solution methodologies\label{sec:Sec3}}

Due to the NP-hardness of  the considered problem, only small size instances
can be solved optimally by the enumeration techniques such as branch
and bound or dynamic programming. However the size of some practical
problem encountered in industry is usually large. We need to develop
 feasible heuristics to deal with large-sized instances. In this section,
a simple weighted search procedure and a general variable neighborhood
search algorithm are proposed.

Before introducing the  algorithms, we discuss two properties
of the considered problem that will be used later.
\begin{property}
\label{pro:Property 1}Consider any two jobs $j$ and $k$ that are
processed before their given deteriorating dates. If  $p_{j}\leq p_{k}$ and $d_{j}\leq d_{k}$, then there is an optimal sequence in which job $j$
is scheduled before job $k$.  \end{property}
\begin{proof}
Since the two jobs under consideration are not deteriorated, their
processing times are only the basic processing times.   Lemma 3.4.1 \citep[p.~51]{pinedo2008scheduling} regarding the total
tardiness problem in a single-machine environment says
if $p_{j}\leqslant p_{k}$ and $d_{j}\leqslant d_{k}$ then there
exists an optimal sequence in which job $j$ is scheduled before job
$k$. Thus, applying this lemma based on the basic processing times and the
due dates  yields an optimal sub-sequence including the job pair ($j$,
$k$) in an optimal schedule. \end{proof}
\begin{property}
\label{pro:Property 2}Consider any two jobs $j$ and $k$ that are
processed after their given deteriorating dates. If
$p_{j}+b_{j}\leq p_{k}+b_{k}$ and $d_{j}\leq d_{k}$, then job $j$
 is scheduled before job $k$   in an optimal sequence. \end{property}
\begin{proof}
Since the two jobs $j$ and $k$ are processed after their deteriorating times,
the actual processing time of each job  is the sum of its basic processing
time and a penalty time. Thus,  like Property \ref{pro:Property 1}, Property \ref{pro:Property 2} is a direct result of  applying Lemma 3.4.1 of \citep{pinedo2008scheduling} to all such job pairs ($j$,
$k$) based on their actual processing times $(a_{j}+b_{j}$, $a_{k}+b_{k}$)
and the due dates ($d_{j}$, $d_{k}$).
\end{proof}
According to the two properties,   a job should be scheduled
in an earlier position if it has a smaller value of $a_{i}$ (or $a_{i}+b_{i}$ if it is deteriorated) and
a smaller value of $d_{i}$. There is a weighted
search algorithm that has  applied to solve the single-machine flow
time scheduling problem \citep{He2009SDPLD}. Thus, we adopt this
idea and provide a simple weighted search procedure (SWSP) to obtain a heuristic for
the $1|p_{j}=a_{j}$ or $a_{j}+b_{j},h_{j}|\sum T_{j}$ problem.

\subsection{Simple weighted search procedure}

We shall need an appropriately chosen triple   $\omega:=(\omega_{1},
\omega_{2}, \omega_{3})$  of positive weights to compute a  weighted value  $m_i:=\omega_{1}d_{i}+\omega_{2}p_{i}+\omega_{3}h_{i}$, for each $i\in N$.
Since it is difficult to determine  a priori  a triple of weights that can yield good
solutions, we adopt the {\em dynamic update strategy}
to search all possible triples  $\omega$.
Weight $\omega_{1}$ is a linearly increasing weight,
given by  $\omega_{1}=\omega_{1\min}+(\omega_{1\max}-\omega_{1\min})(l_{1}-1)/(n-1)$,
where $l_{1}$ varies from 1 to $n$. Weight $\omega_{2}$ adopts
the same updating formula  as     weight $\omega_{1}$. The only
difference between the two weights is their possible minimal and maximal
  values. Suitable values of $\omega_{1\min}$, $\omega_{1\max}$, $\omega_{2\min}$ and $\omega_{2\max}$ may be determined by preliminary experiments. For example,   for the numerical experiments we conducted in this article, we have chosen the  values of the four parameters
$\omega_{1\min}$, $\omega_{1\max}$, $\omega_{2\min}$ and $\omega_{2\max}$ to be
0.2, 0.9, 0.1 and 0.7, respectively. With $\omega_1$ and $\omega_2$ determined,  $\omega_{3}:=1-\omega_{1}-\omega_{2}$.
But $\omega_{3}$ may become negative    by using this formula. Should a negative
$\omega_{3}$ occurs, it is adjusted to  a pre-specified positive value.
In this article, this preset value is chosen to be 0.1. Let $\Omega$  be the set of  all possible
triples of weights   obtained.

We now describe the SWSP. Firstly, let a  triple $(\omega_1, \omega_2,\omega_3)\in \Omega$ be fixed. Then the  job with the minimal due date is
scheduled in the first position. Subsequently the unscheduled jobs are arranged in a nondecreasing
order of the weighted sums $m_i$, $i\in N$. Thus each triple $(\omega_1, \omega_2,\omega_3)\in \Omega$ is associated with a schedule.  Correspondingly, for each schedule, the value of the objective total tardiness is calculated.   Among all these different solutions, the best solution is given by the one with the smallest total tardiness. Furthermore,
in order to refine the quality of the solution,
a local improvement approach based on pairwise swapping of jobs is used in the procedure. The detailed
  SWSP is described in Algorithm \ref{alg:SWSP}.

\begin{algorithm}[h]
\begin{algorithmic}[1]

\caption{\label{alg:SWSP}Simple weighted search procedure (SWSP)}

\STATE Input the initial data of a given instance;

\STATE Set $c_{[0]}=0$, $N_{0}=[1,\cdots,n]$ and $\pi=[\;]$. Set $\pi_{\mathrm{best}}=\pi$ and $Z_{\mathrm{best}}=\sum_{j\in N_0}(a_j+b_j)$;

\STATE Generate the entire set $\Omega$ of possible triples of weights. For each triple $(\omega_{1}, \omega_{2}, \omega_{3})\in \Omega$,
perform the steps 4--27.

\STATE Choose  job $i$ with minimal due date to be scheduled in
the first position;

\STATE $\:$$c_{[1]}=c_{[0]}+a_{i}$, $\pi=\left[ \pi,\; i\right]$;

\STATE $\:$delete  job $i$ from $N_{0}$;

\STATE $\:$set $k=2$;

\REPEAT

\IF {$c_{[k-1]}>\max\{h_{r},r\in N_{0}\}$}

\STATE choose  job $j$ from $N_{0}$ with the smallest weighted value $m_j$ to be scheduled in the $k$th position;
\STATE $c_{[k]}=c_{[k-1]}+a_{j}+b_{j}$, $\pi=\left[ \pi, \;j\right]$;
\STATE delete  job $j$ from $N_{0}$;
\STATE $k=k+1$;

\ELSE

\STATE choose  job $j$ from $N_{0}$ with the smallest value $m_j$ to
be scheduled in the $k$th position;
\IF { $c_{[k-1]}>h_{j}$}

\STATE $c_{[k]}=c_{[k-1]}+a_{j}+b_{j}$;

\ELSE

\STATE $c_{[k]}=c_{[k-1]}+a_{j}$;

\ENDIF

\STATE $\pi=\left[ \pi,\;j\right]$;

\STATE delete  job $j$ from $N_{0}$;

\ENDIF

\UNTIL {the set $N_{0}$ is empty}

\STATE calculate the   total tardiness $Z(\pi)$ of the obtained schedule $\pi$;
\IF {$Z(\pi)<Z_{\mathrm{best}}$}
\STATE $\pi_{\mathrm{best}}=\pi$ and $Z_{\mathrm{best}}=Z(\pi)$;
\ENDIF

\FOR {$i\leftarrow1$ to $n$}

\FOR {$j\leftarrow1$ to $n$}

\IF {$i\neq j$}

\STATE create a new solution $\pi{'}$ by interchanging jobs
in the $i$th and $j$th position from $\pi_{\mathrm{best}}$;

\STATE replace $\pi_{\mathrm{best}}$ by $\pi{'}$ if the total tardiness of $\pi{'}$
is smaller than that of $\pi_{\mathrm{best}}$;

\ENDIF

\ENDFOR

\ENDFOR

\STATE Output the finial solution $\pi_{\mathrm{best}}$.
\end{algorithmic}
\end{algorithm}

To illustrate the proposed heuristic SWSP, a simple example is given with eight jobs. The related data for this instance is provided in Table 1. Suppose that the initial sequence is $N_0=[1,2,3,4,5,6,7,8]$.

   Firstly, the weighted search process is carried out. Using the chosen parameters of $\omega_{1\min}$, $\omega_{1\max}$, $\omega_{2\min}$, $\omega_{2\max}$ and the methods for calculating $\omega_1, \omega_2, \omega_3$ as described above,  an initial triple of weights is determined as (0.2, 0.1, 0.7). Then  job 2 with the minimal due date is assigned to the first position. We then have that  $c_{[1]}=c_{[0]}+a_2=44$, and $\pi=[2]$.
Delete job 2   from $N_0$. Subsequently, the weighted values $m_{j}, j \in N_0\setminus\{2\}$ of unscheduled jobs are calculated and arranged in a nondecreasing order. Then a schedule is obtained with job 2 in the first position followed  by other jobs appearing in the nondecreasing order of the weighted sum values.  In this sequence, for a job at position greater than or equal to 2,  if its deteriorating date is earlier than the completion time of its immediately previous job, then its processing time is the sum of basic processing time and penalty. Otherwise, its processing time is just the  basic processing time.   The resulted  processing sequence by using the triple (0.2, 0.1, 0.7) of weights is   $\pi=[2, 8, 3, 4, 6, 5, 1, 7]$. Its   corresponding total tardiness is calculated as 1291. Next, other triples of weights are generated by the described formulas. For each triple of weights, the previous procedure is iterated to find a new schedule. If the new schedule is better than the incumbent one, the incumbent one is updated. When all possible combinations of the triple weights have been searched, the weighted search process is finished. The sequence delivered by the weighted search is found to be $\pi=[2, 3, 1, 5, 8, 4, 7, 6]$, and the corresponding objective value is 696.

Secondly, the pairwise swap operation based on the sequence delivered by the weighted search is carried out. After the swap operation, the schedule is improved in terms of the total tardiness. The final schedule is given by  $\pi=[3, 2, 4, 1, 5, 7, 8, 6]$ and its objective value is 575 (see figure \ref{fig_NEGC}). We note that the  result given by the SWSP is very close to the optimal solution 572 that is found by CPLEX.

\begin{figure}[h]
  \centering
  \includegraphics[scale=0.95]{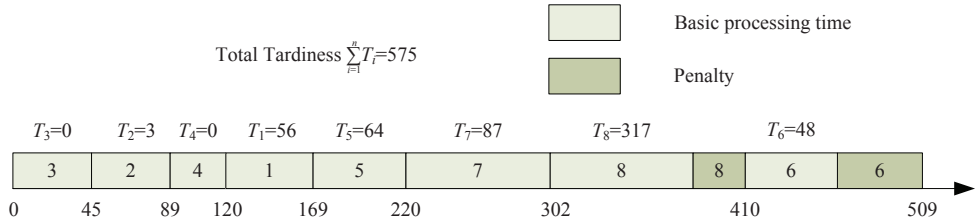}\\
  \caption{Gantt Chart of the schedule given by SWSP}\label{fig_NEGC}
\end{figure}

\begin{table}[htbp]
  \centering
  \caption{The input date of the instance with eight jobs}
    \begin{tabular}{lrrrrrrrr}
    \toprule
    Job($j$) & 1     & 2     & 3     & 4     & 5     & 6     & 7     & 8 \\
    \midrule
    Basic processing time $a_j$ & 49    & 44    & 45    & 31    & 51    & 52    & 82    & 80 \\
    Due date $d_j$ & 113   & 86    & 114   & 218   & 156   & 461   & 215   & 93 \\
    Deteriorating date $h_j$ & 271   & 255   & 91    & 131   & 205   & 101   & 367   & 85 \\
    Penalty $b_j$ & 33    & 19    & 41    & 27    & 18    & 47    & 44    & 28 \\
    \bottomrule
    \end{tabular}%
  \label{tab:addlabel}%
\end{table}%

\subsection{General variable neighborhood search heuristic}

The variable neighborhood search (VNS) is a simple and effective meta-heuristic
proposed by \citet{Mladenovic1997VNS}. It has the advantage of avoiding
entrapment at a local optimum by systematically swapping   neighborhood structures during the search. The basic VNS combines two
search approaches: a stochastic approach in the {\em shaking step}
that finds a random neighbor of the incumbent, and a deterministic
approach which applies  any type of local
search algorithm. Serval VNS variants
have been proposed and successively applied to numerous combinatorial
optimization problems \citep{Hansen2010VNS}. Among which, the variable neighborhood
descent (VND) \citep{Hansen2001VND} as the best improved deterministic
 method is the most frequently used variant. The VND initiates
from   a  feasible solution as the incumbent one and then carries out
a series of neighborhood searches through operators $\mathcal{N}_{k}$,($k=1,\cdots,k_{\max}$).
If a better solution is found in a neighborhood, the incumbent solution is replaced by the better one
and the search  continues within the current neighborhood;
otherwise it will explore the next neighborhood.
The obtained  solution at the end of the search is a local optimum with respect to
all neighborhood structures.  If the deterministic local search approach of the basic VNS is replaced
by the  VND search, the resulted algorithm is called a {\em general
VNS} (GVNS).
The GVNS has  received attentions and showed good performance compared
to other VNS variants. So far, the GVNS has been applied to many optimization
problems, such as the incapacitated single allocation $p$-hub median
problem \citep{Aleksandar2010GVNS}, the one-commodity pickup-and-delivery
traveling salesman problem \citep{Mladenovic2012GVNS}, the capacitated
vehicle routing problem \citep{Lei2012GVNS} and the single-machine
total weighted tardiness problem with sequence dependent setup times
\citep{Kirlik2012GVNS}.

As far as  we know, there is no published work on
solving   scheduling problems with deteriorating jobs by the GVNS.
Therefore, in this section {\em we develop a GVNS heuristic for the single-machine total tardiness scheduling
problem with step-deterioration.}
In what follows,  the main components of the designed GVNS algorithm are described:
the initialization process,   five neighborhood structures provided
for the shaking step and local search procedures. In addition, a
perturbation phase with a 3-opt operation without change of direction
is embedded within the search
process in order to  decrease the probability of returning to the previous local optimum.

\subsubsection{Initialization}

Typically, a good
initial sequence can be obtained   by using the Earliest Due Date
(EDD) rule. That is to say, all jobs are arranged in nondecreasing
order of the due dates.

\subsubsection{Neighborhood structures}

In a local search algorithm, a neighborhood structure is designed by
introducing moves from one solution to another.  In order to conduct a local search in the proposed GVNS,
we next develop five neighborhood structures based on swap, insertion
and $k$-opt for the problem $1|p_{j}=a_{j}$ or $a_{j}+b_{j},h_{j}|\sum T_{j}$.
These neighborhood structures are defined by their corresponding operators. We remark that for simplicity  a neighborhood structure shall be referred to by the name of its  corresponding generating operator. For example, if a neighborhood structure is generated by an operator $\mathcal{N}_1$, the neighborhood structure shall be called $\mathcal{N}_1$ neighborhood structure or simply $\mathcal{N}_1$ neighborhood.

Swap Operator ($\mathcal{N}_{1}$): The swap operator (Figure \ref{fig:fig_swap}) selects a pair of jobs $\pi_{i}$ and $\pi_{j}$ in the current sequence $\pi$ of jobs,
exchanges their positions, and computes the total tardiness of the resulting solution.
If an improvement is found, the new sequence is accepted as the incumbent one.
The operator  repeats this process for all indices $i$,$j\in N$, with $i \neq j$ until all  the neighborhoods have been searched.
The size of the swap operator is $n(n-1)/2$. %%That is, the time complexity of the operator is $O(n^{2})$.
%However, for any given sequence with $n$ jobs,
%it takes $O(n)$ time to calculate the total tardiness. Thus the whole neighborhood operator needs $O(n^{3})$ time to complete the search.

\begin{figure}
  \centering
  \includegraphics[scale=0.6]{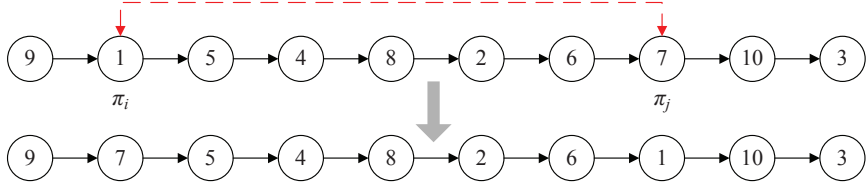}\\
  \caption{Illustration of the swap operator}\label{fig:fig_swap}
\end{figure}

Insertion Operator ($\mathcal{N}_{2}$): For a given incumbent arrangement of jobs, the insertion neighborhood (Figure \ref{fig:fig_insert}) can be obtained
by removing a job from its position and inserting it into another position.
%%That is to say, for a given pair of indices $i$ and $j$ ($1\leqslant i,j\leqslant n$),
%%let $\pi =[\alpha,\pi(i),\beta,\pi(j),\gamma]$ if $i<j$ and $\pi=[\alpha,\pi(j),\beta,\pi(i),\gamma]$ if $i>j$.
 All jobs are considered for insertion operation.
Once an improvement is observed, the new arrangement of jobs is deemed as the incumbent one.
%%This operator also requires $O(n^{2})$ time.

\begin{figure}
  \centering
  \includegraphics[scale=0.6]{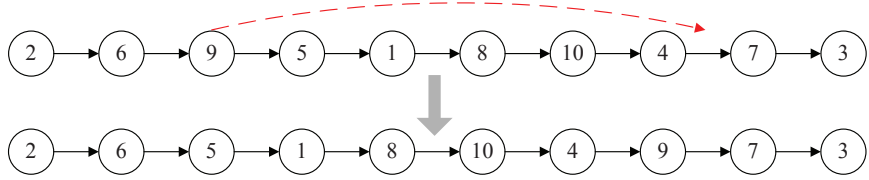}\\
  \caption{Illustration of the insertion operator}\label{fig:fig_insert}
\end{figure}

Pairwise Exchange Operator ($\mathcal{N}_{3}$): The pairwise exchange operator is similar to the swap operator except that the swap is applied to a pair of jobs rather than a single job. In our implementation, these pairwise jobs are selected from all the combinations of two out of $n$ jobs. It exchanges their positions and computes the total tardiness of the resulting sequence. Once an improvement is obtained, the incumbent solution is updated by the new solution.
%%The whole procedure is described in Algorithm \ref{alg:alg4}. %%This operator also runs in $O(n^{2})$ time.

%
%\begin{algorithm}[h]
%\caption{\label{alg:alg4}Pairwise Exchange Moves ($N_{3}$)}
%\begin{algorithmic}[1]
%
%\STATE $i\leftarrow 1$;
%\WHILE {$i\leqslant n$}
%\STATE $j\leftarrow 1$
%\WHILE {$j\leqslant n$}
%\STATE select two distinct positions $k_1$ and $k_2$ that are different from $i$ and $j$;
%\STATE exchange the positions of $i$ and $k_1$ and of $j$ and $k_2$;
%\IF {the new solution is improved}
%\STATE update the current solution;
%\STATE $j\leftarrow n$
%\STATE $i\leftarrow 0$
%\ENDIF
%\STATE $j\leftarrow j+1$;
%\ENDWHILE
%\STATE $i\leftarrow  i+1$;
%\ENDWHILE
%
%\end{algorithmic}
%\end{algorithm}

Couple Insertion Operator ($\mathcal{N}_{4}$): This procedure is similar to the
insertion operator $\mathcal{N}_{2}$. But the   operation is for a pair of successive jobs.
For each couple of jobs $\pi_{i}$ and $\pi_{i+1}$, $1\le i<n-1$, the operator extracts these two jobs
and inserts them in another pair of positions $j$ and $j+1$, $1\le j\le n-1$. Note that $i\neq j$.
%The search procedure repeats until no improvement is obtained. ({\bf what do you mean by `until no improvement is obtained'? Does this operator exhaust all possibilities??? })

2-opt Operator ($\mathcal{N}_{5}$): The 2-opt is the most classical heuristic
for the traveling salesman problem in which it removes two edges from the
tour and reconnects the two paths created. It can be employed as a simple heuristic to solve some scheduling problems \citep{Chang2011229}. In our implementation  (Algorithm
\ref{alg:alg_2opt}), the 2-opt operator selects two jobs $\pi_i$
and $\pi_j$ in the current sequence. It then deletes the edge connecting job $\pi_i$
and its successor and the edge connecting  job $\pi_j$ with its successor.
Afterwards, it  constructs a new connection of $\pi_i$ to $\pi_j$ and a new connection between  their
respective successors. Furthermore, the partial sequence between the successors
of $\pi_i$ and $\pi_j$ is reversed. Figure \ref{fig:fig_2opt} illustrates an
example of the 2-opt procedure. When an improvement is found in terms of the objective value, the incumbent is updated with the improved sequence. The search continues with applying the 2-opt operator to all possible pairs of jobs that are at least {\em three} positions away from each other.  The solution obtained by this operator
may not  significantly differ from the incumbent in terms of the objective values.  But
some random jumps in the objective value may be  achieved to escape from a current local optimum due to the path reversion.

\begin{algorithm}[h]
\caption{\label{alg:alg_2opt}2-opt Operator ($\mathcal{N}_{5}$)}
\begin{algorithmic}[1]

\FOR { $i\leftarrow1$ to $n$ }

\FOR { $j\leftarrow1$ to $n$ }

\STATE $I\leftarrow \min(i,j)$;

\STATE $J\leftarrow \max(i,j)$;

\IF {$J-I\geqslant3$ }

\STATE apply the 2-opt procedure to the jobs $\pi_i$ and $\pi_j$ to generate a new sequence $\pi '$;

\IF {the new sequence $\pi '$ is better than the current sequence in terms of objective value}

\STATE update the incumbent;

\ENDIF
\ENDIF
\ENDFOR
\ENDFOR
\end{algorithmic}
\end{algorithm}

\begin{figure}

\begin{centering}
\includegraphics[scale=0.6]{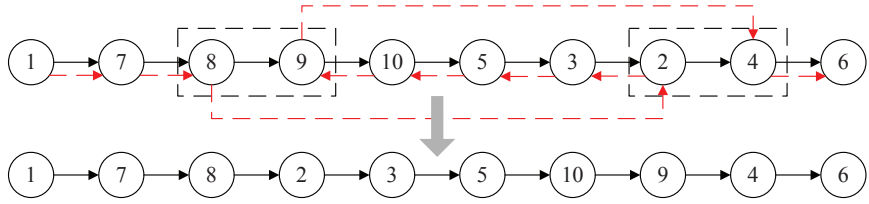}
\par\end{centering}

\caption{\label{fig:fig_2opt}Illustration of the 2-opt operator ($\mathcal{N}_{5}$)}

\end{figure}

\subsubsection{Shaking and local search}
A shaking  operation is performed before the local search in each iteration. The shaking procedure plays a role in diversifying the search and facilitating escape from a local optimum.
To preserve the simplicity of the  principles of the VNS, both  the shaking procedure  and the local search     of the GVNS
make use of the same set $\mathcal{N}_{k}$ ($k$=1,$\ldots,5$) of neighborhood structures. The shaking operation is implemented by generating randomly a neighboring solution $\pi'$ of the current one $\pi$ using a    given neighborhood operator $\mathcal{N}_{k}$. The neighborhood operator $\mathcal{N}_k$ will be chosen   to cycle from $\mathcal{N}_1$ through $\mathcal{N}_5$. If  the given neighborhood structure is $\mathcal{N}_{1}$ or $\mathcal{N}_{5}$, then the shaking procedure randomly selects two jobs. If the given neighborhood structure is $\mathcal{N}_{2}$, then the shaking procedure randomly selects a job and an insertion position. If the given neighborhood structure is $\mathcal{N}_{3}$, then the shaking procedure randomly selects two pairs of jobs. If the given neighborhood structure is $\mathcal{N}_{4}$, then the shaking procedure randomly selects a couple of two consecutive jobs and a couple of two consecutive positions.

A complete local search is organized as a VND
using the proposed neighborhood structures. To efficiently explore
possible solutions, a sequential order $K$ of applying these neighborhoods are randomly generated.
For instance, assuming that the sequence is $K=(3,1,2,5,4)$, the search starts from $\mathcal{N}_{3}$ and ends at $\mathcal{N}_{4}$.
For each neighborhood, a new local optimum $\pi''$ is obtained by
carrying out the corresponding local search operation. If $\pi''$ is better
than $\pi'$, the new solution $\pi''$  will be accepted as a  descent so that  $\pi'$ is updated with $\pi''$,
and the search continues for a new $\pi''$ within the current neighborhood; otherwise the search turns to the next neighborhood in $K$. The search stops until   the last neighborhood  in $K$ is explored. It is worthwhile to point out that after one iteration of a shaking and local search, the   solution $\pi'$ generated by the shaking procedure may not be improved.

%Moreover,
%in an attempt to reduce the computational time for the GVNS, we transpose pairwise adjacent jobs
%if they are not in the specific order mentioned in properties  \ref{pro:Property 1} and   \ref{pro:Property 2} at the end of each iteration of  the shaking and local search steps.  ({\bf  we discussed to remove this paragraph and the item 6 below. Do you decide to keep this????? })
The shaking and local search  steps are summarized below.
%in Algorithm
%\ref{alg:alg5}.
\begin{enumerate}
  \item Begin the shaking procedure. Use the current solution $\pi$ to randomly generate a neighbor $\pi '$  and set $i=1$.
  \item Obtain the sequential order $K$ of neighborhood search at random.
  \item Apply the neighborhood operator $\mathcal{N}_{K(i)}$ to $\pi{'}$ to achieve a local optimum $\pi{''}$, where $K(i)$ means the $i$-th entry of the sequence $K$.
  \item If $\pi{''}$ is better than $\pi{'}$, update $\pi{'}$ and keep the value $i$ so the search will continue within the current neighborhood; otherwise, $i\leftarrow (i+1)$.
  \item If $i>k_{\max}=5$, complete one iteration of the shaking and local search; otherwise, go to Step 3.
 % \item Transpose pairwise adjacent jobs in $\pi'$ if they are not in the specific order mentioned in property \ref{pro:Property 1} and property \ref{pro:Property 2}.
\end{enumerate}

%\begin{algorithm}
%\caption{\label{alg:alg5}Variable neighborhood descent}
%\begin{algorithmic}
%
%\STATE use the incumbent solution $\pi$ to obtain a random solution $\pi^{'}$
%at the shaking phase;
%
%\STATE $k\leftarrow1$;
%
%\REPEAT
%
%\STATE apply the neighborhood $N_{k}$ with $\pi^{'}$ as initial solution
%to generate a local optimal solution $\pi^{''}$;
%
%\IF {the local optimum $\pi^{''}$ is better than $\pi$ }
%
%\STATE continue the search with $N_{k}$ ($k\leftarrow1$);
%
%\ELSE
%\STATE $k\leftarrow k+1$;
%\ENDIF
%\UNTIL {$k=k_{max}$}
%\end{algorithmic}
%\end{algorithm}

\subsubsection{Perturbation phase}

Since a local search applied to optimization problems often suffers  from
getting trapped in a local optimum, the well-know approach for this
deficiency is to adopt a multi-start method when no improvement is
observed. Multi-start heuristics can usually be characterized as iterative
procedures consisting of two phases: the first phase in which a starting
solution is generated and the second one in which the starting solution
is typically (but not necessarily) improved by local search methods.
However, a far more effective approach is to generate the new starting
solution from an obtained local optimum by a suitable perturbation method. %Moreover, Crossover operator is one of three main
%procedures in genetic algorithm. It can produce some promising solutions
%in the solution space for combinatorial optimization problems.

Inspired by the multi-start strategy, we provide a  perturbation phase with 3-opt operator {\em without change of direction}
for producing a new starting solution in case a local search is trapped in a local optimum.
In our implementation, if the current solution $\pi$ cannot be
further improved through a predetermined number $\gamma$ of iterations of the shaking and local search procedures,
the GVNS algorithm assumes that there is no hope to continue the local search
based on the current best solution $\pi $. Then, the 3-opt operator (Figure \ref{fig:fig_3opt}) is implemented on   $\pi$. Three jobs are randomly selected from the incumbent sequence.  The 3-opt operator removes the three edges connecting
the selected jobs with their successors. Then it reconnects the four
subsequences created. The direction of these subsequences remains
unchanged because reversing the direction of these subsequences may produce   a much worse solution after a number of  local searches.

%a randomly
%solution $\pi_{rand}$ is generated and the linear order crossover
%(LOX) is carried out between $\pi_{rand}$ and $\pi_{best}$. A example
%of LOX is illustrated in Figure \ref{fig:fig3}. The offspring solutions
%combined properly some good features of the solution $\pi_{beest}$
%and may have even better features compared to the randomly generated
%solutions. The one with lower objective value among the generated
%offspring solutions is looked as the new starting solution. The GVNS
%is restarted to gain further improvements in the best solutions found
%so far.

%3-opt Moves ($N_{5}$): In order to enough explore the solution space,
%we introduce 3-opt without change of direction to generate more complex
%neighbors. Three jobs are randomly selected from the incumbent sequence,
%let $i$, $j$ and $k$ be these jobs. It remove the three edges connecting
%the selected jobs with their successors. Then It reconnects the four
%subsequences created. The direction of these subsequences remains
%unchanged. This step is repeated $n$ times and an example can found
%in Figure \ref{fig:fig2}.

\begin{figure}

\begin{centering}
\includegraphics[scale=0.6]{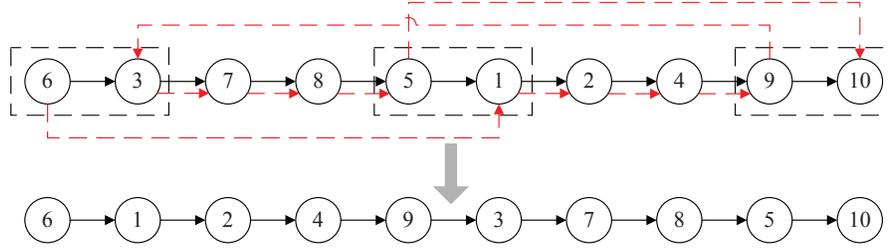}
\par\end{centering}

\caption{\label{fig:fig_3opt}Illustration of the 3-opt moves ($\mathcal{N}_{5}$)}

\end{figure}

\subsubsection{Proposed GVNS algorithm}

%The proposed GVNS algorithm steps are shown in Algorithm \ref{alg:GVNS}.
The GVNS algorithm starts with an initialization phase in which  an initial solution $\pi_0$ is generated by the EDD rule and is set as the current solution $\pi$.
To evaluate the quality of solutions, the solution value is considered as
the total tardiness.

At each iteration, the proposed algorithm needs to carry out mainly
two steps. First,
   a shaking procedure is performed to generate a random neighboring
solution $\pi{'}$ of the incumbent solution $\pi$.
Subsequently, a local search based on the VND with $\pi{'}$ as the input solution
is performed in an attempt to improve $\pi'$.
The local   solution $\pi{'}$, possibly improved by the local search  is then compared to
the current best solution $\pi$ in terms of the objective function value. If $\pi{'}$
is better than $\pi$, the current solution $\pi$ is updated by
$\pi{'}$. This completes one iteration of the shaking and local search. The GVNS algorithm then continues with next  iteration  of shaking and local search.   In addition, if no
  improvement of the current best solution $\pi$ can be found through a predetermined number $\gamma$ of
iterations of shaking and local search, a  3-opt perturbation procedure   is
used to generate a newly starting solution $\pi$.

We next describe the {\em stopping criterion}.
A pre-specified maximum number $Iter_{\mathrm{max}}$
of iterations of shaking and local search  or  a maximum number $Iter_{\mathrm{nip}}$ between two iterations without improvement   are used as the stopping criterion of the algorithm.
In our implementation, we choose  $Iter_{\mathrm{max}}=500$ and $Iter_{\mathrm{nip}}=150$, respectively. Moreover, the predetermined number $\gamma$ of iterations to perform a 3-opt perturbation is chosen to be $\gamma=0.5\cdot Iter_{\mathrm{nip}}$. This strategy is  determined by  our preliminary experiments.
  A Pseudo-code of the proposed GVNS algorithm is summarized in Algorithm \ref{alg:GVNS}.
%   ({\bf it seems that the criteria $\beta$ either can not be reached or the 3-opt operation only can be done once if $\beta$ is used.   } )
%Finally,
%The GVNS algorithm
%continues until when the incumbent solution cannot be further improved
%for a given number of iterations or a pre-specified maximum number
%of iterations is achieved. When the process of GVNS algorithm stops,
%the finial current solution is used as the best solution $\pi^{*}$.

\begin{algorithm}[h]
\caption{\label{alg:GVNS}General variable neighborhood search heuristic}
%({\bf Double check this algorithm!!!!  })
\begin{algorithmic}[1]

\STATE Initialize the parameters of the algorithm;
\STATE Define a set of neighborhood structures $\mathcal{N}_{k}$($k=1,\ldots,5$),
that will be used in the shaking phase and the
local search phase;

\STATE Generate the initial solution $\pi_{0}$ by the EDD rule;

\STATE Calculate the objective value $f(\pi_{0})$ for the solution $\pi_{0}$;

\STATE Set  the current best solution $\pi=\pi_{0}$;

\STATE Choose the stopping criterion; initialize the counter: $iter_1=0$, $iter_2=0$ and $iter_3=0$;

\STATE Set the first neighborhood structure for the shaking procedure to be $k\leftarrow 1$;

\REPEAT

\STATE \textbf{Shaking}: Generate a point $\pi{'}$ at random in
the $\mathcal{N}_{k}$  neighborhood of $\pi$;
\STATE Produce a  random    sequence $K$ of applying the 5 neighborhood structures;

\STATE \textbf{Local search}: Apply the VND scheme in the order  specified by $K$, update $\pi{'}$ if a better local optimum is obtained;

\IF {$f(\pi{'})<f(\pi)$}
\STATE $\pi\leftarrow\pi{'}$
\STATE reset the counter $iter_2=0$;
\STATE reset the counter $iter_3=0$;
\ELSE \STATE Increment the counters: $iter_2=iter_2+1$, $iter_3=iter_3+1$;
\ENDIF

\STATE Update the counter of total number of iterations $iter_1=iter_1+1$;
\IF {$\pi$ has not improved for a given number of iterations $\gamma$, that is, when  $iter_3 >\gamma$, and if $iter_2<iter_{NIP}$  }

\STATE use the 3-opt perturbation procedure to generate a new
starting solution $\pi$; and reset $iter_3=0$;

\ENDIF

\STATE Set $k=k\mod 5 +1$ to cycle through the neighborhood structures for shaking.

\UNTIL {the stopping criterion is met, that is, $iter_1>Iter_{\mathrm{max}}$ or $iter_2>Iter_{\mathrm{nip}}$};

\STATE Output the current best solution $\pi$.
\end{algorithmic}
\end{algorithm}

\section{Computational experiments\label{sec:Sec4}}

In this section, in order to evaluate the performance of the GVNS
algorithm and the proposed heuristic SWSP, a set of testing instances were  generated and solved on
a PC with Intel Core i3 3.20 GHz CPU and 4 GB of memory
in the environment of Windows 7 OS. The procedure of generating the testing instances  and
analysis of the results are described below.

\subsection{Experiments design }

In this article, the basic processing times ($a_{j}$) are randomly
generated from a discrete uniform distribution on the  interval (0, 100{]}.
The deteriorating dates ($h_{j}$) are   picked from the uniform
distributions over three different intervals (0, $A$/2{]}, {[}$A$/2,
$A${]} and (0, $A${]}, where,  the value of $A$ is obtained from
the equation $A=\sum_{j\in N}a_{j}$. The deterioration penalties $b_j, j\in N$,
are   drawn from the uniform distribution (0, 100\texttimes{}$\tau${]},
where we choose $\tau$=0.5. In addition, the due dates are randomly generated
according to a discrete uniform distribution from two different
intervals (0, 0.5\texttimes{}$\bar{C}_{\max}${]},
and (0,$\bar{C}_{\max}${]}, where $\bar{C}_{\max}$ is the value of
the makespan obtained by arranging the jobs in the non-decreasing
order of the ratios $a_{j}/b_{j}$, $j\in N$   \citep{Bachman2000LD}.

The algorithms were tested over five different job sizes for  small-sized instances. Those sizes are respectively $n$= 8, 10, 15, 20 and 25. For all possible combinations
of deteriorating dates and due dates, to account for the three intervals in which deteriorating dates are generated and the two intervals in which due dates are drawn,  six  types of problems
are needed to solve for a specific job size. For convenience, these
types of problems are denoted by symbols $S$\_$k_{1}k_{2}$.
For example, $S$\_12 represents a type of problems with deteriorating
dates drawn from the  interval (0, $A$/2{]} and due dates drawn from the  interval (0,$\bar{C}_{\max}${]}.
In addition,   a set of large sized instances with $n$=\{50,
60, 70, 80, 90, 100\} are being considered. As a consequence, for both small-sized and large-sized problems,
66 ($(5+6)\times6$) sample  instances were randomly generated.

\subsection{Experimental results}

Because there are no comparative data and no competing heuristic for
our problem, comparisons with best know solutions are not possible.
Thus, we have designed and coded the standard VNS heuristic (Algorithm \ref{alg:VNS})
as a comparison to assess our approach. The VNS heuristic does not include the perturbation phase.
 The implementation sequence of neighborhoods is chosen by the deterministic order
$( \mathcal{N}_{1}, \mathcal{N}_{2}, \mathcal{N}_{3}, \mathcal{N}_{4}$ and $\mathcal{N}_{5})$.  The stopping criteria  are chosen to be  identical to the GVNS algorithm.

\begin{algorithm}[h]
\caption{\label{alg:VNS}Variable neighborhood search heuristic}
%{\bf THis algorithm has some problem. Please double check this. THe problem is what is one iteration? and How the stopping criterion is applied.   }
\begin{algorithmic}[1]
\STATE Initialize the parameters of the algorithm;
\STATE Define a set of neighborhood structures $\mathcal{N}_k$,$k=1,\ldots,5$;
\STATE Generate an initial solution $\pi_{0}$  by the EDD rule;
\STATE Set   the current best solution $\pi=\pi_{0}$;
\STATE Initialize the counters: $iter_1=0$, $iter_2=0$.
\STATE $k\leftarrow 1$;
\REPEAT
%\WHILE {$k\leqslant 5$}
\STATE \textbf{Shaking}: Generate randomly a point $\pi{'}$ in the $k$th neighborhood of $\pi$;
\STATE \textbf{Local search}: Apply neighborhood search operator $\mathcal{N}_{k}$ with $\pi{'}$ as initial solution to obtain a local optimum  $\pi{''}$;
\IF {$f(\pi{''})<f(\pi)$}
\STATE $\pi\leftarrow\pi{''}$;
\STATE continue the local search with the current neighborhood $\mathcal{N}_{k}$;
\STATE reset the counter $iter_2=0$;
\ELSE
\STATE  $k\leftarrow k \mod 5+1$;
\STATE Increment the counter: $iter_2=iter_2+1$;
\ENDIF
%%\ENDWHILE
%\IF {$f(\pi)<f(\pi^{*})$}
%\STATE $\pi^{*}\leftarrow\pi$;
%
%\ELSE
%
%\ENDIF
\STATE Update the counter of total number of iterations $iter_1=iter_1+1$;
\UNTIL {the stopping criterion is reached}, that is, $iter_1>Iter_{\max}$ or $iter_2>Iter_{\mathrm{nip}}$;
\STATE Output the current best solution $\pi$.
\end{algorithmic}
\end{algorithm}

The described problem can be solved optimally by a commercial solver
CPLEX 12.5 for some small-sized instances. However, because of its NP-hardness,
it is impossible to obtain optimal solution by the CPLEX
for medium or large instances. The preset run time for CPLEX is one hour. If CPLEX fails to either
converge to the optimum or prove the optimality of the incumbent within
the time limit, a GAP is provided by the software. The GAP shows the quality of a solution given by the CPLEX within the run time limit to some extent.
A large GAP implies that the CPLEX needs more time to converge to an optimal solution.
% ({\bf The explanation of the GAP does not make sense to me. It should indicate the distance from a certain kind of lower bound or best solution??????  }  )

All proposed algorithms are implemented in Matlab 7.11.
The CPLEX and the SWSP are deterministic hence only one run is necessary for each problem instance. While the VNS and the GVNS are stochastic so we have to run some replications in order to better assess the experimental results. In our experiment, for each instance we run 10 times when using the VNS and the GVNS.

Each algorithm's performance is measured   by computing a relative percentage deviation ($RPD$) defined by the equation
\begin{gather}\label{eq:4.1}
 RPD(\%)=\frac{Z_{\mathrm{alg}}-Z_{\mathrm{best}}}{Z_{\mathrm{best}}}\times100
\end{gather}
where $Z_{\mathrm{alg}}$ is the solution value obtained for a given algorithm and instance, $Z_{\mathrm{best}}$ is the best solution of all approaches.
%The performance of these heuristics is compared in terms of RPD and computational time over small- and large-sized instances.
 For the VNS and the GVNS, the comparison results  are  in terms of average $RPD$ and average computational time.

Table \ref{tab:small} summaries the results obtained from the CPLEX, the SWSP, the VNS and the GVNS. As shown in the table, the GVNS finds the best solutions for all small-sized instances. Since the computational time of the CPLEX was preset to one hour, the CPLEX could not find the optimal solutions  for all instances  with 20 or 25 jobs. The computational time of the CPLEX exponentially increases as the number of jobs increases
due to the NP-hardness of the problem.
Note that 4 solutions obtained by the CPLEX are worse than those  obtained by the  GVNS.
The average $RPD$ given  by the SWSP is 11.63$\%$. The run times of the SWSP are very short and are ignored in table \ref{tab:small}. Compared to the CPLEX, the run times of the GVNS and the VNS do not significantly increase as the number of jobs increases.  The  computational time of the GVNS is longer than that of the VNS because the GVNS includes the VND scheme as the local search and additionally the perturbation phase. However, the VNS delivers the best solutions for only 13 out of the 30 small-sized test instances.
 The $RPD$ delivered by VNS is only 2.32$\%$. This suggests that these designed neighborhood structures are well suitable for solving the problem under consideration.

%The results indicate that the GVNS is very effective in solving the single-machine scheduling problem with step-deteriorating jobs.

\begin{sidewaystable*}[h]
  \centering
  %\begin{threeparttable}[b]
  \caption{Comparison of the   CPLEX, the SWSP, the VNS and  the GVNS for small-sized instances.} \label{tab:small}%
    \begin{tabular}{ccccccccccccc}
    \toprule
    \multirow{2}[1]{*}{Groups} & \multirow{2}[1]{*}{\textit{n}} & \multicolumn{3}{c}{CPLEX} & \multicolumn{2}{c}{SWSP} & \multicolumn{3}{c}{VNS} & \multicolumn{3}{c}{GVNS} \\

    \cmidrule[0.05em](r){3-5}
    \cmidrule[0.05em](lr){6-7}
    \cmidrule[0.05em](lr){8-10}
    \cmidrule[0.05em](l){11-13}

        &       & Objective Value & Time(s) & GAP(\%) & Objective Value & RPD(\%)   & Mean  & RPD(\%) & Time(s) & Mean  & RPD(\%) & Time(s) \\

    \midrule

    \emph{S}\_11   & 8     & 585   & 0.12  & 0.00  & 585   & 0.00  & 585   & 0.00  & 0.08  & \textbf{585} & 0.00  & 0.75  \\
          & 10    & 638   & 0.87  & 0.00  & 680   & 6.58  & 645   & 1.10  & 0.12  & \textbf{638} & 0.00  & 1.44  \\
          & 15    & 1891  & 1621.10  & 0.00  & 2129  & 12.59  & 1896.2 & 0.27  & 0.37  & \textbf{1891} & 0.00  & 3.75  \\
          & 20    & 3062  & 3600.00  & 60.79  & 3427  & 11.92  & 3095.8 & 1.10  & 0.74  & \textbf{3062} & 0.00  & 7.70  \\
          & 25    & 5476  & 3600.00  & 79.40  & 6048  & 16.17  & 5236.4 & 0.58  & 1.49  & \textbf{5206} & 0.00  & 16.63  \\
    \emph{S}\_12   & 8     & 301   & 0.08  & 0.00  & 301   & 0.00  & 301   & 0.00  & 0.07  & \textbf{301} & 0.00  & 0.68  \\
          & 10    & 570   & 0.30  & 0.00  & 595   & 4.39  & 590.8 & 3.65  & 0.12  & \textbf{570} & 0.00  & 1.40  \\
          & 15    & 739   & 57.00  & 0.00  & 772   & 4.47  & 739   & 0.00  & 0.25  & \textbf{739} & 0.00  & 3.24  \\
          & 20    & 343   & 482.67  & 0.00  & 366   & 6.71  & 366   & 6.71  & 0.56  & \textbf{343} & 0.00  & 6.81  \\
          & 25    & 1062  & 3600.00  & 28.61  & 1757  & 65.44  & 1066.5 & 0.42  & 1.04  & \textbf{1062} & 0.00  & 13.81  \\
    \emph{S}\_21   & 8     & 464   & 0.09  & 0.00  & 464   & 0.00  & 464   & 0.00  & 0.07  & \textbf{464} & 0.00  & 0.61  \\
          & 10    & 609   & 0.45  & 0.00  & 615   & 0.99  & 609   & 0.00  & 0.11  & \textbf{609} & 0.00  & 1.05  \\
          & 15    & 1243  & 463.31  & 0.00  & 1248  & 0.40  & 1243  & 0.00  & 0.22  & \textbf{1243} & 0.00  & 2.84  \\
          & 20    & 2929  & 3600.00  & 67.15  & 3161  & 7.92  & 2929  & 0.00  & 0.55  & \textbf{2929} & 0.00  & 7.11  \\
          & 25    & 5057  & 3600.00  & 96.03  & 5036  & 2.71  & 4909  & 0.12  & 0.62  & \textbf{4903} & 0.00  & 11.17  \\
    \emph{S}\_22   & 8     & 189   & 0.11  & 0.00  & 194   & 2.65  & 189.5 & 0.26  & 0.06  & \textbf{189} & 0.00  & 0.59  \\
          & 10    & 569   & 1.93  & 0.00  & 569   & 0.00  & 569   & 0.00  & 0.11  & \textbf{569} & 0.00  & 1.06  \\
          & 15    & 540   & 120.70  & 0.00  & 540   & 0.00  & 540   & 0.00  & 0.25  & \textbf{540} & 0.00  & 3.03  \\
          & 20    & 1352  & 3600.00  & 21.45  & 1352  & 0.00  & 1352  & 0.00  & 0.49  & \textbf{1352} & 0.00  & 5.95  \\
          & 25    & 191   & 3600.00  & 11.52  & 302   & 58.12  & 203.6 & 6.60  & 0.87  & \textbf{191} & 0.00  & 9.83  \\
    \emph{S}\_31   & 8     & 664   & 0.09  & 0.00  & 672   & 1.20  & 664   & 0.00  & 0.07  & \textbf{664} & 0.00  & 0.73  \\
          & 10    & 856   & 1.28  & 0.00  & 858   & 0.23  & 856   & 0.00  & 0.13  & \textbf{856} & 0.00  & 1.22  \\
          & 15    & 2316  & 833.79  & 0.00  & 2437  & 5.22  & 2323.1 & 0.31  & 0.37  & \textbf{2316} & 0.00  & 3.56  \\
          & 20    & 2332  & 3600.00  & 44.51  & 2547  & 9.31  & 2331.4 & 0.06  & 0.73  & \textbf{2330} & 0.00  & 8.09  \\
          & 25    & 2354  & 3600.00  & 80.46  & 2503  & 6.33  & 2360.5 & 0.28  & 0.98  & \textbf{2354} & 0.00  & 12.15  \\
    \emph{S}\_32   & 8     & 319   & 0.06  & 0.00  & 319   & 0.00  & 319   & 0.00  & 0.07  & \textbf{319} & 0.00  & 0.68  \\
          & 10    & 217   & 0.45  & 0.00  & 236   & 8.76  & 232.2 & 7.00  & 0.11  & \textbf{217} & 0.00  & 1.17  \\
          & 15    & 50    & 23.65  & 0.00  & 80    & 60.00  & 62.4  & 24.80  & 0.41  & \textbf{50} & 0.00  & 3.73  \\
          & 20    & 378   & 1106.25  & 0.00  & 500   & 32.28  & 378.5 & 0.13  & 0.79  & \textbf{378} & 0.00  & 5.92  \\
          & 25    & 218   & 3600.00  & 34.40  & 234   & 24.47  & 218.2 & 16.06  & 0.86  & \textbf{188} & 0.00  & 11.16  \\
    \multicolumn{2}{c}{Average} &       & 1357.14  & 17.48  &       & 11.63  &       & 2.32  & 0.42  &       & 0.00  & 4.93  \\
    \bottomrule

    \end{tabular}%
%    \begin{tablenotes}
%    \item [a] {}
%    \end{tablenotes}
  %\end{threeparttable}
\end{sidewaystable*}%
%\clearpage

For large-sized instances, it is impossible to find the optimal solution by using the CPLEX within the preset one-hour time limit. Thus we tested the performance of the proposed algorithms with respect to the standard VNS. The results are shown in table \ref{tab:large}. The average $RPD$ of the GVNS is about 0.78$\%$ which is significantly smaller than that of the VNS. Except for  instance $S\_12$ with 60 jobs, the  RPDs of other instances given by the GVNS are very small. The relative higher RPD of   instance $S\_{12}$ may be  because its optimal solution value is relatively small and there are 2 local optima found in 10 replications. According to figure \ref{RPD_L}, it is observed that {\em the quality of solutions delivered by the GVNS is very good without respect to the distribution of the deteriorating dates}.
On the other hand, the quality of solutions obtained by the SWSP is inferior compared to the GVNS and the VNS.

To measure the robustness of the algorithm, a mean absolute deviation (MAD) of 10 runs   when applying  the VNS or  the GVNS was calculated for each of  large-sized instances according to the equation
\begin{gather}\label{eq:4.2}
 MAD(\%)=\frac{1}{R\times \bar Z_{\mathrm{alg}}}\sum_{r=1}^{R}(Z_{\mathrm{alg}}(r)-\bar{Z}_{\mathrm{alg}})\times100
\end{gather}
where $R$ is the number  of replications,  $\bar{Z}_{\mathrm{alg}}$ denotes the average solution value obtained from the  $R$ runs for a given algorithm, and $Z_{\mathrm{alg}}(r)$ indicates the solution value obtained for a given algorithm in the $r$-th run.
The average $MAD$ of the GVNS is merely 0.81$\%$.
Figure \ref{MAD_L} reports the $MAD$s obtained by the GVNS for large-sized instances.
Overall, the GVNS demonstrates to be  a good alternative to solve  single-machine total tardiness scheduling problems with step-deteriorating jobs, because the algorithm finds good solutions in a reasonable computational time.

%constitutional

\begin{sidewaystable*}[h]
  \centering
  %\begin{threeparttable}[b]
  \caption{Comparison of the SWSP, the VNS and the GVNS for large-sized instances.} \label{tab:large}%
    \begin{tabular}{ccccccccccccc}
    \toprule
    \multicolumn{1}{c}{\multirow{2}[1]{*}{Groups}} & \multicolumn{1}{c}{\multirow{2}[1]{*}{\textit{n}}} & \multicolumn{2}{c}{SWSP} & \multicolumn{4}{c}{VNS}       & \multicolumn{4}{c}{GVNS} \\

    \cmidrule[0.05em](r){3-4}
    \cmidrule[0.05em](lr){5-8}
    \cmidrule[0.05em](l){9-12}

    \multicolumn{1}{c}{} & \multicolumn{1}{c}{} & \multicolumn{1}{c}{Objective Value} & \multicolumn{1}{c}{RPD(\%)} & \multicolumn{1}{c}{Mean} & \multicolumn{1}{c}{RPD(\%)} & \multicolumn{1}{c}{MAD(\%)} & \multicolumn{1}{c}{Time(s)} & \multicolumn{1}{c}{Mean} & \multicolumn{1}{c}{RPD(\%)} & \multicolumn{1}{c}{MAD(\%)} & \multicolumn{1}{c}{Time(s)} \\

    \midrule

   \multicolumn{1}{c}{\emph{S}\_11} & \multicolumn{1}{c}{50} & 20820 & 15.90  & 18200 & 1.31  & 0.58  & 5.67  & 17971 & 0.04  & 0.06  & 120.13  \\
    \multicolumn{1}{c}{} & \multicolumn{1}{c}{60} & 23895 & 15.24  & 20995 & 1.25  & 0.89  & 13.17  & 20761 & 0.13  & 0.15  & 224.42  \\
    \multicolumn{1}{c}{} & \multicolumn{1}{c}{70} & 42027 & 19.00  & 36339 & 2.90  & 0.67  & 14.86  & 35324 & 0.02  & 0.01  & 338.22  \\
    \multicolumn{1}{c}{} & \multicolumn{1}{c}{80} & 67739 & 24.61  & 55266 & 1.66  & 0.42  & 20.30  & 54400 & 0.07  & 0.09  & 533.01  \\
    \multicolumn{1}{c}{} & \multicolumn{1}{c}{90} & 82112 & 16.35  & 72307 & 2.46  & 0.44  & 25.02  & 70675 & 0.14  & 0.04  & 634.76  \\
          & \multicolumn{1}{c}{100} & 78452 & 17.05  & 67609 & 0.87  & 0.26  & 58.67  & 67215 & 0.28  & 0.19  & 931.67  \\
    \multicolumn{1}{c}{\emph{S}\_12} & \multicolumn{1}{c}{50} & 8501  & 29.27  & 6817.3 & 3.67  & 2.18  & 5.63  & 6603.1 & 0.41  & 0.74  & 114.08  \\
    \multicolumn{1}{c}{} & \multicolumn{1}{c}{60} & 1208  & 4546.15  & 437.3 & 1581.92  & 29.66  & 13.76  & 29.6  & 13.85  & 17.03  & 145.59  \\
    \multicolumn{1}{c}{} & \multicolumn{1}{c}{70} & 4047  & 289.13  & 2279.4 & 119.17  & 8.48  & 16.69  & 1065.4 & 2.44  & 2.22  & 274.37  \\
    \multicolumn{1}{c}{} & \multicolumn{1}{c}{80} & 9808  & 56.20  & 7365.1 & 17.30  & 6.44  & 20.89  & 6318.6 & 0.63  & 0.98  & 440.14  \\
          & \multicolumn{1}{c}{90} & 8731  & 110.13  & 5015.7 & 20.71  & 5.89  & 37.52  & 4212  & 1.37  & 1.08  & 625.96  \\
    \multicolumn{1}{c}{} & \multicolumn{1}{c}{100} & 4961  & 158.92  & 3211.2 & 67.60  & 9.51  & 48.56  & 1966.8 & 2.65  & 3.05  & 851.83  \\
    \multicolumn{1}{c}{\emph{S}\_21} & \multicolumn{1}{c}{50} & 20980 & 6.28  & 19788 & 0.24  & 0.18  & 5.19  & 19740 & 0.00  & 0.00  & 69.45  \\
    \multicolumn{1}{c}{} & \multicolumn{1}{c}{60} & 29281 & 5.93  & 27694 & 0.19  & 0.09  & 8.19  & 27642 & 0.00  & 0.00  & 116.10  \\
    \multicolumn{1}{c}{} & \multicolumn{1}{c}{70} & 46270 & 8.77  & 42636 & 0.23  & 0.15  & 14.47  & 42554 & 0.04  & 0.05  & 209.57  \\
          & \multicolumn{1}{c}{80} & 50660 & 10.42  & 46103 & 0.49  & 0.23  & 15.67  & 45880 & 0.00  & 0.00  & 225.04  \\
    \multicolumn{1}{c}{} & \multicolumn{1}{c}{90} & 57859 & 13.55  & 51118 & 0.32  & 0.11  & 29.56  & 50958 & 0.00  & 0.01  & 406.09  \\
    \multicolumn{1}{c}{} & \multicolumn{1}{c}{100} & 80275 & 11.98  & 71975 & 0.40  & 0.18  & 30.83  & 71695 & 0.01  & 0.01  & 583.66  \\
    \multicolumn{1}{c}{\emph{S}\_22} & \multicolumn{1}{c}{50} & 320   & 57.64  & 212.1 & 4.48  & 5.15  & 4.77  & 203   & 0.00  & 0.00  & 43.82  \\
    \multicolumn{1}{c}{} & \multicolumn{1}{c}{60} & 1647  & 55.23  & 1198.3 & 12.94  & 9.48  & 7.88  & 1061  & 0.00  & 0.00  & 114.46  \\
          & \multicolumn{1}{c}{70} & 5063  & 8.91  & 4797.4 & 3.19  & 1.55  & 11.64  & 4649  & 0.00  & 0.00  & 152.83  \\
    \multicolumn{1}{c}{} & \multicolumn{1}{c}{80} & 2535  & 34.98  & 2161.5 & 15.10  & 3.05  & 17.96  & 1878  & 0.00  & 0.00  & 219.78  \\
    \multicolumn{1}{c}{} & \multicolumn{1}{c}{90} & 11089 & 17.18  & 9785.6 & 3.41  & 1.64  & 33.16  & 9463.1 & 0.00  & 0.00  & 444.46  \\
    \multicolumn{1}{c}{} & \multicolumn{1}{c}{100} & 3467  & 23.34  & 3007.8 & 7.00  & 1.45  & 24.15  & 2811  & 0.00  & 0.00  & 363.21  \\
    \multicolumn{1}{c}{\emph{S}\_31} & \multicolumn{1}{c}{50} & 15639 & 14.35  & 13817 & 1.03  & 0.35  & 7.56  & 13676 & 0.00  & 0.00  & 91.41  \\
          & \multicolumn{1}{c}{60} & 31578 & 9.49  & 28930 & 0.31  & 0.06  & 7.76  & 28842 & 0.00  & 0.00  & 143.20  \\
    \multicolumn{1}{c}{} & \multicolumn{1}{c}{70} & 38983 & 19.39  & 32953 & 0.92  & 0.33  & 19.41  & 32660 & 0.02  & 0.03  & 290.32  \\
    \multicolumn{1}{c}{} & \multicolumn{1}{c}{80} & 49429 & 27.75  & 39330 & 1.65  & 0.41  & 27.62  & 38725 & 0.08  & 0.06  & 398.79  \\
    \multicolumn{1}{c}{} & \multicolumn{1}{c}{90} & 75215 & 17.39  & 64855 & 1.22  & 0.63  & 32.61  & 64123 & 0.08  & 0.06  & 623.88  \\
    \multicolumn{1}{c}{} & \multicolumn{1}{c}{100} & 76221 & 21.44  & 63412 & 1.03  & 0.28  & 30.48  & 62765 & 0.00  & 0.00  & 641.97  \\
    \multicolumn{1}{c}{\emph{S}\_32} & \multicolumn{1}{c}{50} & 1175  & 48.92  & 1121  & 42.08  & 3.07  & 5.51  & 789   & 0.00  & 0.00  & 64.15  \\
          & \multicolumn{1}{c}{60} & 999   & 239.80  & 688.7 & 134.25  & 8.09  & 8.67  & 301.2 & 2.45  & 0.96  & 113.51  \\
          & \multicolumn{1}{c}{70} & 3867  & 94.03  & 2914.6 & 46.24  & 7.29  & 15.94  & 1994  & 0.05  & 0.05  & 244.40  \\
          & \multicolumn{1}{c}{80} & 22127 & 21.74  & 18782 & 3.34  & 1.35  & 22.28  & 18179 & 0.02  & 0.02  & 517.98  \\
          & \multicolumn{1}{c}{90} & 7729  & 45.34  & 5791.5 & 8.90  & 2.17  & 31.23  & 5408  & 1.69  & 1.20  & 647.39  \\
          & \multicolumn{1}{c}{100} & 5061  & 162.64  & 4207.2 & 118.33  & 7.05  & 34.63  & 1955.8 & 1.49  & 1.18  & 540.83  \\
    \multicolumn{2}{c}{Average} &       & 174.29  &       & 61.89  & 3.33  & 20.22  &       & 0.78  & 0.81  & 347.23  \\
    \bottomrule
    \end{tabular}%
%    \begin{tablenotes}
%    \item [a] {}
%    \end{tablenotes}
  %\end{threeparttable}
\end{sidewaystable*}%
%\clearpage

\begin{figure}[h]
  \centering
  \includegraphics[scale=0.4]{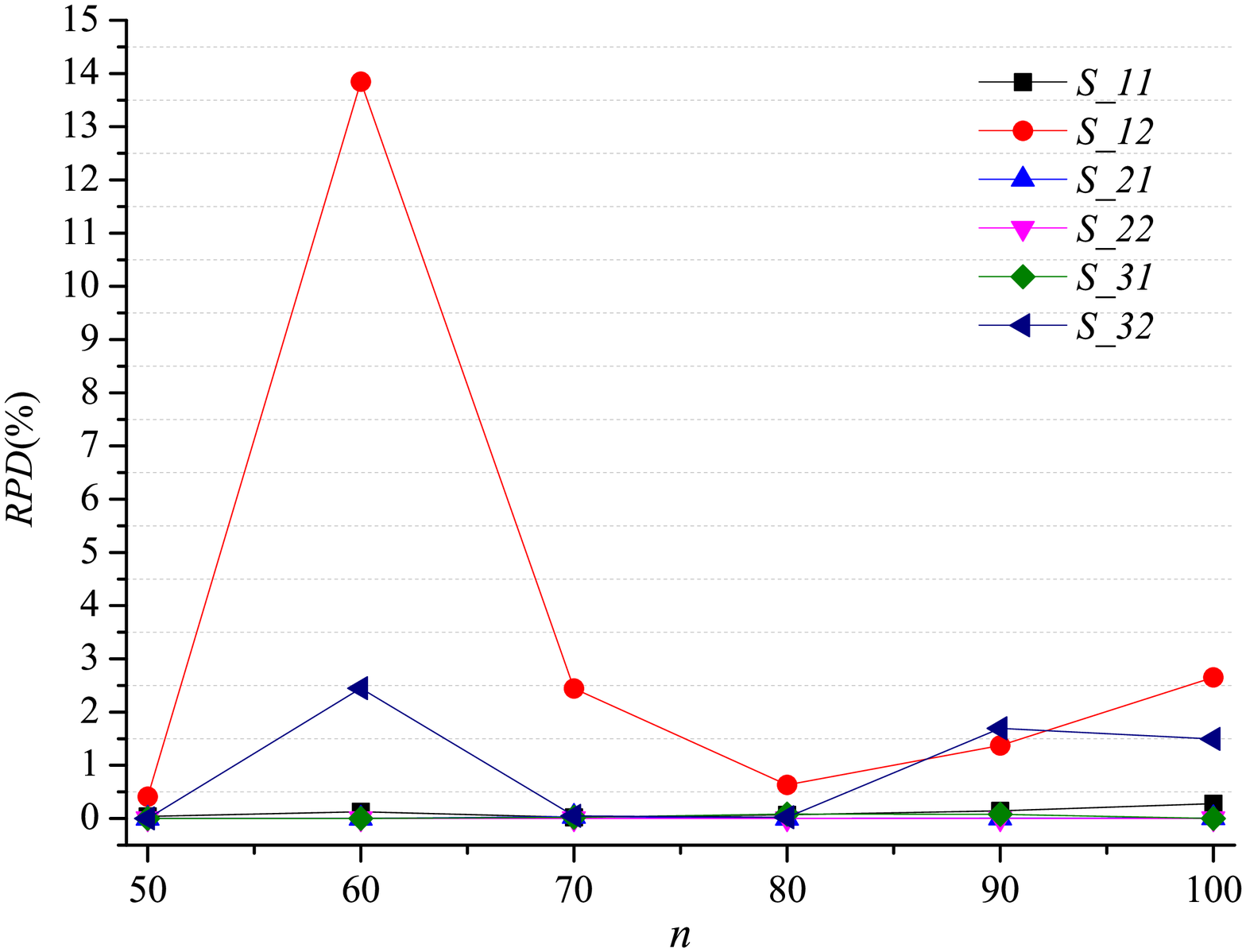}\\
  \caption{The $RPD$s given by the GVNS for   large-sized instances}\label{RPD_L}
\end{figure}

\begin{figure}[h]
  \centering
  \includegraphics[scale=0.4]{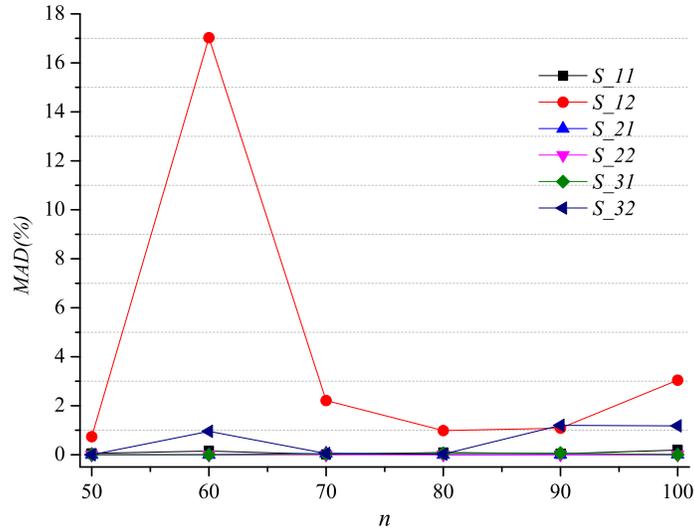}\\
  \caption{The $MAD$s given by the GVNS for   large-sized instances}\label{MAD_L}
\end{figure}

\section{Conclusions}
In this article, we considered a single-machine total tardiness scheduling problem
with step-deteriorating jobs. The objective of this problem is to determine the sequence policy of the jobs under consideration so as to minimize the total tardiness. In order to solve the problem, an integer programming model was developed. Solutions   were obtained by the CPLEX up to the size of  25 jobs. Due to the intractability of the problem,
it is impossible to solve large instances to optimality by using the CPLEX.
Based on two properties introduced, a heuristic called the SWSP and the GVNS algorithm were proposed to obtain   near-optimal solutions of the problem.
Meanwhile, the standard VNS algorithm was presented as reference for evaluating the performance of our algorithms. Experimental results
showed that the proposed GVNS outperforms  other heuristics in terms of relative percentage deviation from the best solution for both small- and large-sized instances. The computational times are in general short for the GVNS and the VNS. The GVNS performed better than the VNS in both cases. Furthermore, the GVNS is more robust than the VNS with regards  to the choice of different intervals of deteriorating dates.
%was capable of obtaining better solutions with a reasonable computational time compared to the two heuristic, especially for large sized instances.

For further study, a suggestion is to consider the setup times between different jobs for the single-machine scheduling problem with step-deteriorating jobs. In addition,  multi-machine or multi-criteria cases with step-deterioration are also encouraged.

%For acknowledgements section, please don't number the section, please begin it with \section*{Acknowledgements}
\section*{Acknowledgments}
This work was partially supported by the National Natural Science
Foundation of China (No.51175442) and the Fundamental Research Funds
for the Central Universities (No. 2010ZT03; SWJTU09CX022).

% You may incorporate your references as follows in your main tex file.
% Using BibTex is not recommended but can be handled.
\bibliographystyle{plainnat}
\addcontentsline{toc}{section}{\refname}\bibliography{SPSDTT}

\medskip
% The data information below will be filled by AIMS editorial staff
Received xxxx 20xx; revised xxxx 20xx.
\medskip

\end{document}